\documentclass[12pt]{amsart}
\usepackage{geometry}
\usepackage{amssymb,latexsym}
\newcommand{\RNum}[1]{\lowercase\expandafter{\romannumeral #1\relax}}
\usepackage{xcolor}
\usepackage{makecell}
\usepackage{amsthm}
\usepackage{amsmath}
\usepackage{hyperref}
\usepackage{mathrsfs}
\hypersetup{colorlinks=true,linkcolor=blue,urlcolor=blue,citecolor=red}
\newtheorem{proof 1}{proof 1}
\newtheorem{theorem}{Theorem}[section]
\newtheorem{example}{Example}[section]

\newtheorem{remark}{Remark}[section]

\newcommand{\beq}{\begin{equation}}
\newcommand{\eeq}{\end{equation}}

\newtheorem{definition}{Definition}[section]
\newtheorem*{equiv*}{Equivalence Relation}

\newtheorem{corollary}{Corollary}[section]

\title[]{Combinatorics of higher order degenerate  $r$-DERANGED Bell Numbers}

\author{Sithembele Nkonkobe}
\address{School of Mathematics, University of Witwatersrand, 2050 Wits, Johannesburg, South Africa}
\email{snkonkobe@gmail.com}

\date{\today}

\begin{document}

	\begin{abstract} When one inserts a number of identical bars in between blocks of an ordered set partition, they obtain a barred preferential arrangement. In this study we define a new generalization of barred preferential arrangements, by considering barred preferential arrangements with no fixed blocks, and where $r$ fixed elements on each section of the barred preferential arrangement are singletons. We derive several combinatorial identities for the number of these barred preferential arrangements. We also provide some asymptotic results.
	   
	\end{abstract}
	\maketitle
	
	\section{Introduction and Preliminaries}
	
Ordered Bell numbers (also known as Fubini numbers) play a central role in combinatorics as they have found applications in various areas of mathematics, science and engineering. For instance, ordered Bell numbers have found application in topology as giving number of chain topologies \cite{stanley1971number}. They also give the number of weak orderings of an $n$-element \cite{bailey1998number}. The numbers have found applications in graph theory on Cayley trees, which are structures used to represent mathematical models\cite{cayley1875analytical}. They also appear in probability theory in connection with moments of a geometric distribution\cite{simsek2025analysis}. The numbers have found applications on a problem about effective resistance of a hypercube where all the connections have resistance of 1 ohm \cite{pippenger2010hypercube}. They have also found applications in probability theory in connection with moments of the negative binomial distribution \cite{adell2023generalized}. In this study, we introduce a new novel generalization of the ordered bell numbers that does not allow fixed blocks. 

  When one inserts bars inbetween blocks of an ordered set partition they obtain a \textit{barred preferential arrangement}. The number of barred preferential arrangements of a set have extensively been studied in \cite{adell2023generalized,JoseNkonkobeunified,adell2023higher,benyi2024combinatorial,corcino2020second,nkonkobe2020combinatorial,nkonkobe2017study}. In a given barred preferential arrangement the bars induce sections onto which ordered blocks of the set $[n]=\{1,2,3,\ldots,n\}$ are formed. In the following example we discuss two examples of barred preferential arrangements.
		
		\begin{example}Two barred preferential arrangements of $[6]$;
			
			\RNum{1}. $6\quad13|\quad|24\quad5$,
			
			\RNum{2}. $34\quad6|12\quad5|\quad|$.
			
			The barred preferential arrangement in \RNum{1} has two bars hence three induced sections. The first section i.e. the section to the left of the first bar (from left to right) has two blocks; $\{6\}$, and $\{1,3\}$. The second section is empty i.e. the section between the two bars. The third section has two blocks $\{2,4\}$, and $\{5\}$. The barred preferential arrangement in \RNum{2} has four sections. The first section has two blocks $\{3,4\}$, and $\{6\}$. The second section has the blocks $\{1,2\}$, and $\{5\}$. Both the third, and the fourth sections of this barred preferential arrangement are empty.
			
		\end{example}
	 Barred preferential arrangements seem to first appear in \cite{ahlbach2013barred,pippenger2010hypercube}. We now define generalized stirling numbers which we will later include in our combinatorial proofs. The generalization $S(n,k,\alpha,\beta,\gamma)$ of the stirling numbers  is defined in \cite{hsu1998unified} in the following way; $(t|\alpha)_n=\sum\limits_{k=0}^{n}S(n,k,\alpha,\beta,\gamma)(t-\gamma|\beta)_k$, where $(t|\alpha)_n$ is the factorial polynomial $(t|\alpha)_n=\prod\limits_{k=0}^{n-1}(t-k\alpha)$, such that $n\geq1$ and $\alpha,\beta,\gamma$ real or complex not all equal to zero. It is known that (see~\cite{hsu1998unified});
\begin{equation}\label{equation:3}
	\frac{[((1+\alpha t)^\frac{\beta}{\alpha}-1)]^k}{\beta^k}(1+\alpha t)^{\frac{\gamma}{\alpha}}=k!\sum\limits_{n=0}^{\infty}S(n,k,\alpha,\beta,\gamma)\frac{t^n}{n!}.
\end{equation}

The generalized stirling numbers $S(n,k,\alpha,\beta,\gamma)$ are a generalization of many other special numbers in mathematics, in the following table we name a few. 
\begin{table}[!h]\caption{Some special cases of\label{equation:60} $S(n,k,\alpha,\beta,\gamma)$}\label{table:1}
\begin{enumerate}
	\item For  $(\alpha,\beta,\gamma)=(0,1,0)$ the numbers  $S(n,k,\alpha,\beta,\gamma)$ give the classical Stirling numbers of the second kind $S(n,k)$ (see \cite{comtet1974advanced}). 
	\item  For $(\alpha,\beta,\gamma)=(0,1,r)$ the numbers  $S(n,k,\alpha,\beta,\gamma)$ give the $r$-Stirling numbers of the second kind $S(n+r,k+r)_r$ (see \cite{broader}). 
	\item For $(\alpha,\beta,\gamma)=(\alpha,1,0)$ the numbers  $S(n,k,\alpha,\beta,\gamma)$ give Carlitz degenerate stirling numbers (see \cite{carlitz}). 
	\item For the case  $(\alpha,\beta,\gamma)=(0,1,-a+b)$ the numbers  $S(n,k,\alpha,\beta,\gamma)$ give the Gould-Hopper’s noncentral Lah numbers \cite{gouldhopper}.
	\item For the case $(\alpha,\beta,\gamma)=(1,0,b-a)$ the numbers  $S(n,k,\alpha,\beta,\gamma)$ give the Riordan’s noncentral Stirling numbers (see \cite{Riordan}).
	\item  For the case $(\alpha,\beta,\gamma)=(0,\beta,r)$ we obtain $r$-Whitney numbers (see \cite{whitney}).
	\end{enumerate}
\end{table}

We note that $(t|1)_n=(t)_n=t(t-1)(t-2)\cdots(t-n+1)$. A combinatorial interpretation of the numbers $S(n,k,\alpha,\beta,\gamma)$ can be found in \cite{benyi2022unfair}.  The next definition gives a combinatorial interpretations of the numbers $k!\beta^kS(n,k,\alpha,\beta,\gamma)$, and the numbers $\beta^kS(n,k,\alpha,\beta,\gamma)$. Going forward when we refer to a cell, we mean a block of a partition, and by compartment we mean a separate part of a given block of a partition. Throughout this study the terms cells and blocks of partitions will be used interchangeably.
\begin{definition}\cite{corcino2001combinatorial}\label{definition:10}
	Let $\gamma$ and $\beta$ be nonnegative integers. Let $\alpha$ be such that ${\alpha}|\beta$ and ${\alpha}|\gamma$. Suppose the following conditions are true:
	\begin{enumerate}
		\item there are $k$ like cells, and 1 unlike cell,
		\item each of the $k$ like cells contains $\beta$ labelled compartments (these cells will be referred to as \textit{ordinary cells}),
		\item  the one unlike cell contains $\gamma$ labelled compartments (the cell will be referred to as the \textit{special cell}),
		\item  in each of the cells, compartments have cyclic ordered numbering,
		\item each compartment may be occupied by at most one element,
		\item in each consecutive available $\alpha$ compartments only the first compartment would be occupied by an element.
	\end{enumerate}
	Then the number of ways of distributing elements of $[n]$ into the compartments satisfying the above conditions, one element at a time, such that the  $k$ like cells are nonempty, is $\beta^kS(n,k,\alpha,\beta,\gamma)$, and the case where the $k$ cells are unlike is $k!\beta^kS(n,k,\alpha,\beta,\gamma)$.
\end{definition}

\begin{definition}\label{definition:110}For $\alpha,\beta,\gamma\in\mathbb{N}_0$ (non-negative integers) such that $\alpha|\beta$, and $\alpha|\gamma$ the numbers $x^{k}\beta^{k}S(n,k,\alpha,\beta,\gamma)$ represent the number of partitions of $[n+r]$  into $k+r+1$ blocks, such that the following conditions are satisfied:
	\begin{enumerate}
		\item of the first $k+r$ blocks, $r$ fixed elements from $[n+r]$ form $r$ singletons, say these are the first $r$ elements of $[n+r]$, these are the only blocks that will not have any compartments,
		\item from the $k+r$ blocks the other $k$ blocks are like and each has $\beta$ labelled compartments, and all the $k$ cells must be none empty, 
			\item the  $k$ blocks are each colored with one of $x$ available colors independently. 
		\item the $(k+r+1)th$ cell has $\gamma$ labelled compartments, this is the only block that may be empty,
	
		%	\item the first $r$ subsets of $[n]$ are in different cells,
		%\item The $k+r$ subsets are  
		%\item there are $k$ like cells, and 1 unlike cell,
		%	\item each of the $k$ like cells contains $\beta$ labelled compartments (these cells will be referred to as \textit{ordinary cells}),
		%	\item  the one unlike cell contains $\gamma$ labelled compartments (the cell will be referred to as the \textit{special cell}),
		\item  in each of the blocks, compartments have cyclic ordered numbering,
		\item each compartment may be occupied by at most one element,
		\item elements are placed into blocks having compartments one at a time,
		\item in each consecutive available $\alpha$ compartments only the first compartment may be occupied by an element.
	\end{enumerate}
\end{definition}

\iffalse
Clearly it follows from \eqref{equation:3}, and Definition~\ref{definition:10} that 	\fontsize{11}{1}
\begin{equation}\label{equation:101}\frac{{[x((1+\alpha t)^\frac{\beta}{\alpha}-1)]^k}(1+\alpha t)^{\frac{\gamma}{\alpha}}[x((1+\alpha t)^\frac{\beta}{\alpha}-1)]^r}{(k+r)!}=\sum\limits_{n=0}^{\infty}x^{k+r}\beta^{k+r}S(n,k+r,\alpha,\beta,\gamma)\frac{t^n}{n!}.\end{equation}
\fi

	The well known ordered Bell numbers $G(n)$ are defined  in the following way
\begin{equation}\label{equation:33}
	\sum\limits_{n=0}^{\infty}G(n)\frac{t^n}{n!}=\frac{1}{2-e^t}. 
\end{equation} Going forward, by Bell numbers we mean ordered Bell numbers. 
Various polynomial generalizations of Bell numbers have been studied in the literature for instance in \cite{adell2023generalized,JoseNkonkobeunified,BoyadzhievDil2016,corcino2020second,gross,kargin2016mellin,kimfubini}. The polynomials $\omega _n^{(\lambda)}(x;{\alpha}, \beta, \gamma)$ which are a kind of generalization of Bell numbers have been extensively studied in \cite{kargin2018higher,kargin2019recurrences}. It is known that their generating function is (see~\cite{kargin2018higher}),
\begin{equation}\label{equation:5}	\sum\limits_{n=0}^{\infty} \omega _n(x;{\alpha}, \beta, \gamma,\lambda)\frac{t^n}{n!}=
	\frac{(1+{\alpha} t)^{\frac{\gamma}{{\alpha}}}}{(1-x((1+{\alpha} t)^{\beta/{\alpha}}-1))^{\lambda}}.\end{equation}
Polynomials related to the polynomials $\omega _n(x;{\alpha}, \beta, \gamma,\lambda)$ have been studied in \cite{mihoubi2017partial}.
\begin{remark}\label{remark:1}
	The total number of ways of distributing $n$ elements into $k+1$ cells satisfying the conditions of Definition~\ref{definition:10} such that each of the first $k$ cells are unlike and are each having $\beta$ compartments, and  each of the $k$ cells is colored with one of $x$ available colours, where $k$ runs from 0 to $n$ is (see~\cite{nkonkobe2020combinatorial});

	\begin{equation}\label{equation:6}
		\omega _n(x;{\alpha}, \beta, \gamma,\lambda)=\sum\limits_{k=0}^{n}\binom{k+\lambda-1}{k}x^kk!\beta^k S(n,k,\alpha,\beta,\gamma).
\end{equation}\end{remark}

\begin{definition}[Derangement\cite{remond1980essai}] Given a permutation $\sigma([n])=\sigma(1)\sigma(2)\cdots\sigma(n)$ of the set $[n]$. The permutation $\sigma$ is said to be a derangement if $\sigma(i)\not=i$ for all $i$, i.e. the permutation has no fixed elements.
	
\end{definition}The number $d_n$ of derangements of $[n]$ is well known in the literature  and has the following well known closed form (for instance \cite{comtet1974advanced,graham1989concrete,wilf2005generatingfunctionology}): \begin{equation}
	d_n=n!\sum\limits_{i=0}^n\frac{(-1)^i}{i!}.
\end{equation}

Consider the partition $\pi=a_1a_2\cdots a_k$ of disjoint subsets of the set $[n]$.  The blocks $a_i$ can be arranged in ascending order based on minimum elements $min\: (a_1)< min\: (a_2)<\cdots<min\: (a_k)$, in this article we will refer to this kind of ordering subsets of a set as the standard form. %An ordered partition $\lambda_{\sigma}$ of $[n]$ is a permutation $\sigma$ of $\pi$. Hence, all permutations of the blocks form distinct ordered partitions $$\lambda_{\sigma}([n])=a_{\sigma(1)}a_{\sigma(2)}\cdots a_{\sigma(k)}.$$
\begin{definition}\label{defintion:20}The $r$-derangement numbers (denoted by $d_{k,r}$)  are the number (see~\cite{wang2017r}) of derangements of the set $[k+r]$ such that the first $r$ elements of $[k+r]$ are in distinct cycles. The $r$-derangement numbers ($d_{k,r}$) have the following generating function (see~\cite{wang2017r});
	\begin{equation}\label{equation:1}
		\sum_{k=0}^{\infty}d_{k,r}\frac{t^k}{k!}=\frac{t^re^{-t}}{(1-t)^{r+1}}.
\end{equation}

The numbers $d_{k,r}$ for $r\in\mathbb{N}$, and $s\in \{1,2,3,\ldots,r\}$ have the following recurrence relation (see \cite{wang2017r})

\begin{equation}\label{equation:88}
	d_{k,r}=\sum\limits_{j=s}^{k}\binom{j-1}{s-1}\frac{k!}{(k-j)!}d_{k-j,r-s}.
\end{equation}\end{definition}

\begin{definition}[Deranged Partition\cite{belbachir2023deranged}] A deranged partition $\lambda_{\sigma}$ of $[n]$ is a derangement $\sigma$ of $\pi=a_1a_2\cdots a_k$,  that is $\lambda_{\sigma}([n])=a_{\sigma(1)}a_{\sigma(2)}\cdots a_{\sigma(n)}$ where ${\sigma(t)}\not={t}$ for all $1\leq t\leq k$.
	
\end{definition}
\begin{definition}[$r$-Deranged Partition\cite{belbachir2023deranged}] An $r$-deranged partition $\lambda_{\sigma}$ of $[n+r]$ is an $r$-derangement $\sigma$ of the set of partitions $a_1a_2\cdots a_{r+k}$,  that means\\ $\lambda_{\sigma}([n+r])=a_{\sigma(1)}a_{\sigma(2)}\cdots a_{\sigma(r)}a_{\sigma(r+1)}\cdots a_{\sigma(k+r)}$ such that $a_{\sigma(t)}\not=a_{t}$ for all $1\leq t\leq r+k$.
	
\end{definition}

\begin{definition}[$r$-Deranged Bell \label{defintion:100} numbers\cite{belbachir2023deranged}]The deranged Bell numbers  (denoted by $\mathcal{B}^{r}_{n}$) combinatorially  (see~\cite{belbachir2023deranged}) give the total number of $r$-deranged partitions of $[n+r]$. 
	
	The numbers  $\mathcal{B}^{r}_{n}$ have the following generating function (see~\cite{belbachir2023deranged});
	
	\begin{equation}\label{equation:2}
		\sum_{n=0}^{\infty}\mathcal{B}^{r}_{n}\frac{t^n}{n!}=\frac{e^{tr}(e^t-1)^re^{-(e^t-1)}}{(2-e^t)^{r+1}}.
\end{equation}\end{definition}

 The numbers   $\mathcal{B}^{r}_{n}$ have the following closed form (see~\cite{belbachir2023deranged});
\begin{equation}
	\mathcal{B}^{r}_{n}=\sum\limits_{i=0}^nd_{i,r}S(n+r,i+r),
\end{equation}

where $S(n+r,i+r)$ are stirling numbers of the second kind. In this study we introduce the following novel generalization of the notion of $r$-deranged Bell numbers $\mathcal{B}^{r}_{n}$
in the following way
\begin{equation}\label{equation:4}
	\sum\limits_{n=0}^{\infty}\mathcal{B}^{r,x}_{n,\lambda}(\alpha,\beta,\gamma)\frac{t^n}{n!}=(1+\alpha t)^{\frac{\gamma}{\alpha}}\frac{[x((1+\alpha t)^\frac{\beta}{\alpha}-1)]^{r\lambda}exp(-\lambda[x((1+\alpha t)^\frac{\beta}{\alpha}-1)])}{[1-x((1+\alpha t)^\frac{\beta}{\alpha}-1)]^{(r+1)\lambda}}.
\end{equation}
We now show few values of the numbers $\mathcal{B}^{r,x}_{n,\lambda}(\alpha,\beta,\gamma)$.
\begin{align*}
	\mathcal{B}^{1,x}_{2,1}(\alpha,\beta,\gamma)&=2x^2\beta^2+2\gamma x\beta+x\beta^2-\alpha x\beta\\
		\mathcal{B}^{2,x}_{3,1}(\alpha,\beta,\gamma)&=6x^2\beta^3+6x^2\beta^2\gamma-8x^2\beta^2\alpha+12x^3\beta^3\\
	\mathcal{B}^{3,x}_{4,1}(\alpha,\beta,\gamma)&=36x^3\beta^4+24x^3\beta^3\gamma-48x^3\beta^3\alpha+72x^4\beta^4.\\
		\mathcal{B}^{4,x}_{5,1}(\alpha,\beta,\gamma)&=240x^4\beta^5+120x^4\beta^4\gamma-312x^4\beta^4\alpha+480x^5\beta^5.
\end{align*}

For the special case where $\alpha$ tends to zero and $(\beta,\gamma,\lambda,x)=(1,r,1,1)$ in \eqref{equation:4} we get the generating function in \eqref{equation:2} of the numbers $\mathcal{B}^{r}_{n}$. The main question this study seeks to answer is the following. How can the numbers $\mathcal{B}^{r,x}_{n,\lambda}(\alpha,\beta,\gamma)$ we defined above be interpreted combinatorially? In section  \ref{section:2} we give a combinatorial interpretation of the numbers $\mathcal{B}^{r,x}_{n,\lambda}(\alpha,\beta,\gamma)$ for the case $\lambda=1$, and in Section~\ref{section:3} we propose an answer to this question for arbitrary $\lambda$ in the set of all natural numbers. In our construction of higher order $r$-deranged Bell numbers we make use of the generalized stirling numbers $S(n,k,\alpha,\beta,\gamma)$, which are a generalization of various kinds of special numbers, as also seen in Table~\ref*{table:1} above. Hence, the combinatorial interpretation we propose in this study may be used for studies on combinatorial properties of Bell like numbers and polynomials, by adapting the arguments we give in this study.   In Section~\ref{section:4} we give some asymptotic results. 

%We will show that the polynomials $\mathcal{B}^{r,x}_{n,\lambda}(\alpha,\beta,\gamma)$ give the number of barred preferential arrangements of $[n]$ satisfying the conditions of Definition~\ref{definition:11}, such that in each partition of $[n]$, the first $r$ blocks (ranked by minimum element) are not in the same cycle.
\section{Degenerate $r$-deranged Bell numbers\label{section:2}}
In this case, we consider $\lambda=1$ in \eqref{equation:4}. For the case $\lambda=1$ we refer to the numbers $\mathcal{B}^{r,x}_{n,1}(\alpha,\beta,\gamma)$ in \eqref{equation:4} as degenerate $r$-deranged Bell numbers. We let $\mathbb{N}$ denote the set of natural numbers, and $\mathbb{N}_0$ denote the set of non-negative integers. What we want to achieve is ordered set partitions with no fixed blocks. In the previous section we considered generalized set partitions whose cardinality is  $k!\beta^kS(n,k,\alpha,\beta,\gamma)$. In order for us to obtain ordered set partitions without fixed blocks we restrict ourselves to the following subset. 
\begin{definition}\label{definition:11}For $\alpha,\beta,\gamma\in\mathbb{N}_0$ such that $\alpha|\beta$, and $\alpha|\gamma$ we denote by $d_{k,r}\beta^{k} S(n,k,\alpha,\beta,\gamma)$	the number of partitions of $[n+r]$ into $k+r+1$ blocks one at a time, such that the following conditions are satisfied:
\begin{enumerate}
	\item of the first $k+r$ blocks, the first $r$ blocks must be singleton blocks formed by the first $r$ elements of $[n+r]$, 
	
	\item the $k$ other blocks are like blocks where each of the $k$ blocks has $\beta$ labelled compartments, and all the $k+r$ blocks must be none empty, the $k$ blocks will be referred to as the $\beta$-blocks,
%	\item when linearly ordering the $k+r$ blocks the $r$ singletons must be in different cycles,
	
\item the $(k+r+1)th$ block has $\gamma$ labelled compartments, this is the only block that may be empty,
	
%	\item the first $r$ subsets of $[n]$ are in different cells,
	%\item The $k+r$ subsets are  
	%\item there are $k$ like cells, and 1 unlike cell,
%	\item each of the $k$ like cells contains $\beta$ labelled compartments (these cells will be referred to as \textit{ordinary cells}),
%	\item  the one unlike cell contains $\gamma$ labelled compartments (the cell will be referred to as the \textit{special cell}),
	\item  in each of the blocks, compartments have cyclic ordered numbering,
	\item each compartment may be occupied by at most one element,
	\item in each consecutive available $\alpha$ compartments only the first compartment may be occupied by an element,
		\item elements are placed into blocks having compartments one at a time,
%	\item the first $r$ elements of $[n+r]$ are in different blocks,
	\item the first $k+r$ blocks are linearly ordered such that the $r$ singletons are in different cycles, and none of the blocks are fixed, where the standard form is used as point of reference.
%	\item $\lambda$ bars are inserted in-between the linearly ordered first $k+r$ blocks to form a barred preferential arrangement.
\end{enumerate}
\end{definition}

We will refer to the partitions in Definition~\ref{definition:11} as  $(x,\alpha,\beta,1,r)$-partitions.
%For the case $(r,x,\alpha,\beta,\gamma)=(0,1,\alpha,\beta,\gamma)$, Equation~\eqref{equation:101} becomes Equation~\eqref{equation:3}.
%Also, comparing \eqref{equation:101} with \eqref{equation:3} we have that 
%\begin{equation}
%	k!x^{k+r}\beta^{k+r}S(n+r,k+r,\alpha,\beta,\gamma)_r=x^{k+r}\beta^{k+r}(k+r)!S(n,k+r,\alpha,\beta,\gamma).
%\end{equation}

\begin{remark}\label{remark:3}Note the difference between the generalized stirling numbers \\ $k!\beta^kS(n,k,\alpha,\beta,\gamma)$ considered in the previous section, and the ones we propose \\$d_{k,r}\beta^{k} S(n,k,\alpha,\beta,\gamma)$	in Definition~\ref{definition:11} above. In both cases we have the terms $\beta^kS(n,k,\alpha,\beta,\gamma)$ which have been well studied in the literature (see for instance \cite{benyi2024combinatorial,kargin2018higher,nkonkobe2020combinatorial}). The only difference is the $k!$ replaced with $d_{k,r}$. It is necessary to replace the $k!$ with the derangement numbers $d_{k,r}$. This is so that we exclude from counting those arrangements that have fixed blocks which we are not interested on in this paper. As we are only interested in partitions with no fixed blocks. You may have noticed that in Definition~\ref{definition:110} reference is made to $[n]$ elements when referring to the numbers $\beta^k S(n,k,\alpha,\beta,\gamma)$, whereas in Definition~\ref{definition:11} above we refer to them with reference to the set $[n+r]$, well when counting we get the same number of unordered partitions $\beta^kS(n,k,\alpha,\beta,\gamma)$. On the combinatorial model in Definition~\ref{definition:11} the first $r$ elements of $[n+r]$ are singletons. The number of outcomes only differs when we start ordering the blocking. 
	
\end{remark}
Definition~\ref{definition:11} only caters for partitions for a fixed value of $k$. To cater for all possible partitions we define the following.

\begin{definition}\label{definition:5}For $\alpha,\beta,\gamma\in\mathbb{N}_0$ such that $\alpha|\beta$, and $\alpha|\gamma$ we let $\mathcal{B}^{r,x}_{n,1}(\alpha,\beta,\gamma)$ denote the total number of all possible  $(x,\alpha,\beta,1,r)$-partitions of the set $[n+r]$, where the blocks 	each having $\beta$ compartments are independently each colored with one of $x$ available colors.  
\end{definition}
\begin{remark}The polynomials $\mathcal{B}^{r,x}_{n,1}(\alpha,\beta,\gamma)$ we defined in Definition~\ref{definition:5} are an extension of the $r$-Deranged Bell numbers $\mathcal{B}^{r}_{n}$ discussed in 
the previous section (also studied in \cite{belbachir2023deranged})). The $r$-Deranged Bell numbers $\mathcal{B}^{r}_{n}$ are the special case $\lim\limits_{\alpha\rightarrow0}\mathcal{B}^{r,1}_{n,1}(\alpha,1,r)=\mathcal{B}^{r}_{n}$. Combinatorially the extension $\mathcal{B}^{r,x}_{n,1}(\alpha,\beta,\gamma)$ of the $r$-Deranged Bell numbers requires that cells to have compartments, and the blocks being colored with $x$ colors which is not the case for the $r$-Deranged Bell numbers $\mathcal{B}^{r}_{n}$. The numbers of partitions we described in Definition~\ref{definition:5} are a subset of those given by the polynomials $\omega _n(x;{\alpha}, \beta, \gamma,1)$ we discussed in the previous section (also studied in \cite{kargin2018higher,kargin2019recurrences}). The number of partitions given by $\omega _n(x;{\alpha}, \beta, \gamma,1)$ are all possible arrangements with and without fixed blocks. As expected $\omega _n(x;{\alpha}, \beta, \gamma,1)=\sum\limits_{k=0}^{n}x^kk!\beta^k S(n,k,\alpha,\beta,\gamma)$ has a $k!$ term. That is the main difference between the two types of polynomials. We will shortly see in Theorem~\ref{theorem:5} that this $k!$ term on the polynomials $\mathcal{B}^{r,x}_{n,1}(\alpha,\beta,\gamma)$ is replaced with the derangement numbers $d_{k,r}$, to ensure we only capture partitions with no fixed blocks.
\end{remark}

In the following theorem we propose a closed form expression for the numbers $\mathcal{B}^{r,x}_{n,1}(\alpha,\beta,\gamma)$.
\begin{theorem} \label{theorem:5}For $\gamma,\beta,r\in\mathbb{N}_0$, and $\alpha\in \mathbb{N}$, where ${\alpha}|\beta$ and ${\alpha}|\gamma$ we have
	%\begin{equation}\label{equation:45}	\sum\limits_{n=0}^\infty\mathcal{B}^{r,x}_{n,1}(\alpha,\beta,\gamma)\frac{t^n}{n!}=(1+\alpha t)^{\frac{\gamma}{\alpha}}\frac{[x((1+\alpha t)^\frac{\beta}{\alpha}-1)]^re^{-[x((1+\alpha t)^\frac{\beta}{\alpha}-1)]}}{(1-[x((1+\alpha t)^\frac{\beta}{\alpha}-1)])^{r+1}}.\end{equation}
		\begin{equation}\label{definition:100}
		\mathcal{B}^{r,x}_{n,1}(\alpha,\beta,\gamma)=\sum_{k=0}^{n}d_{k,r}x^{k}\beta^{k} S(n,k,\alpha,\beta,\gamma).
	\end{equation}
\end{theorem}
\begin{proof} By Equation~\eqref{equation:3} we have the following generating function for the generalized stirling numbers $S(n,k,\alpha,\beta,\gamma)$, \begin{equation}\label{equation:322}
		\frac{[((1+\alpha t)^\frac{\beta}{\alpha}-1)]^k}{\beta^k}(1+\alpha t)^{\frac{\gamma}{\alpha}}=k!\sum\limits_{n=0}^{\infty}S(n,k,\alpha,\beta,\gamma)\frac{t^n}{n!}.
	\end{equation}
	
	Also, from Definition~\ref{equation:1} we have the following generating function for the $r$-derangement numbers $d_{k,r}$, 
		\begin{equation}\label{equation:323}
		\sum_{k=0}^{\infty}d_{k,r}\frac{t^k}{k!}=\frac{t^re^{-t}}{(1-t)^{r+1}}.
	\end{equation}

	Equation~\eqref{equation:4} implies the following,

	\begin{equation}\label{equation:324}(1+\alpha t)^{\frac{\gamma}{\alpha}}\frac{[x((1+\alpha t)^\frac{\beta}{\alpha}-1)]^re^{-[x((1+\alpha t)^\frac{\beta}{\alpha}-1)]}}{(1-[x((1+\alpha t)^\frac{\beta}{\alpha}-1)])^{r+1}}=\sum\limits_{k=0}^{\infty}d_{k,r}\frac{[x((1+\alpha t)^\frac{\beta}{\alpha}-1)]^k}{k!}(1+\alpha t)^{\frac{\gamma}{\alpha}}. \end{equation}

	To see why \eqref{equation:324} is correct we make the substitution $t=x(1+\alpha t)^\frac{\beta}{\alpha}-1$ into \eqref{equation:323}, and compare with \eqref{equation:4}. Equation~\ref{equation:324} implies 
	\begin{align*}
		\sum\limits_{n=0}^\infty	\mathcal{B}^{r,x}_{n,1}(\alpha,\beta,\gamma)\frac{t^n}{n!}=&\sum\limits_{k=0}^{\infty}d_{k,r}\sum\limits_{n=0}^{\infty}\beta^{k}x^{k}S(n,k,\alpha,\beta,\gamma)\frac{t^n}{n!}\\=&\sum\limits_{n=0}^{\infty}\begin{bmatrix}\sum\limits_{k=0}^nd_{k,r}x^{k}\beta^{k}S(n,k,\alpha,\beta,\gamma)\end{bmatrix}\frac{t^n}{n!}.
	\end{align*}
	
	Thus, 
		\begin{equation}
		\mathcal{B}^{r,x}_{n,1}(\alpha,\beta,\gamma)=\sum_{k=0}^{n}d_{k,r}x^{k}\beta^{k} S(n,k,\alpha,\beta,\gamma).
	\end{equation}
	
\end{proof}

We now provide a combinatorial proof of Theorem~\ref{theorem:5}.

\begin{proof}The $(x,\alpha,\beta,1,r)$-partitions of $[n+r]$ can be formed in the following way. We partition $[n+r]$ into $k+r+1$ subsets such that the first $r$ elements of $[n+r]$ are singletons, $k$ other blocks are each having $\beta$ compartments, and one block has $\gamma$ compartments. The number of these partitions of $[n+r]$ into the $k+r+1$ blocks is $\beta^{k} S(n,k,\alpha,\beta,\gamma)$  (see Definition~\ref{definition:10}). The $k$ blocks having $\beta$ compartments each can be colored independently with $x$ colors in $x^k$ ways. The number of ways such that the $k+r$ blocks other than the block having $\gamma$ compartments may be linearly ordered  such that the $r$ singletons are in different cycles, and none of the blocks are fixed is given by the $r$-derangement numbers in $d_{k,r}$ (see Definition~\ref{defintion:20}). Taking the product and summing over $k$ to get the total number of $(x,\alpha,\beta,1,r)$-partitions we get $	\mathcal{B}^{r,x}_{n,1}(\alpha,\beta,\gamma)=\sum\limits_{k=0}^{n}d_{k,r}x^{k}\beta^{k} S(n,k,\alpha,\beta,\gamma)$.   
\end{proof}

\section{The general case: $\lambda$ an arbitrary positive integer\label{section:3}}

In this section we discus the general case in \eqref{equation:4} i.e. $\lambda$ being an arbitrary positive integer. We propose a combinatorial interpretation of the numbers $\mathcal{B}^{r,x}_{n,\lambda}(\alpha,\beta,\gamma)$. As a generalization of the combinatorial interpretation we have given of the numbers $\mathcal{B}^{r,x}_{n,1}(\alpha,\beta,\gamma)$, we give the following combinatorial interpretation of the numbers  $\mathcal{B}^{r,x}_{n,\lambda}(\alpha,\beta,\gamma)$.

\begin{definition}\label{definition:270}For $\alpha,\beta,\gamma\in\mathbb{N}_0$ such that $\alpha|\beta$, and $\alpha|\gamma$ we denote by $\mathcal{B}^{r,x}_{n,\lambda}(\alpha,\beta,\gamma)$	the number of partitions of $[n+r\lambda]$ into $\lambda+1$ sections where on the $ith$ section there are $\lambda_i+r$ elements such that $\lambda_1+\lambda_2+\cdots+\lambda_{\lambda+1}=n$  such that the following conditions are satisfied:
	
\begin{itemize}\item 	elements going into the first $\lambda$ sections satisfy the following conditions
	\begin{enumerate}
		\item as each section has $\lambda_i+r$ (for $i=1,2,3,\ldots,\lambda$) elements we will refer to the $r$ elements as special elements, 
		\item in each section the $r$ special elements are fixed, hence we have $r\lambda$ fixed elements, $r$ of them per section,
	%	\item from the set  $[n+r\lambda]$ it is only the $n$ other elements that may belong to any of the sections,
		\item on each section the $r$ elements form $r$ singleton blocks,
		\item on the $ith$ section elements are partitioned into $k_i+r$ blocks, where the first $r$ blocks are singleton blocks formed by the special $r$ elements,
		%\item when linearly ordering the $k+r$ blocks the $r$ singletons must be in different cycles, 
		\item the $k_i$ blocks each having $\beta$ compartments,
		\item each of the $k_i$ blocks is nonempty,
		\item each of the $k_i$ blocks is colored independently with one of $x$ available colors, 
		\item elements are placed into each of the $k_i$ blocks one at a time,
		
		%	\item the first $r$ subsets of $[n]$ are in different cells,
		%\item The $k+r$ subsets are  
		%\item there are $k$ like cells, and 1 unlike cell,
		%	\item each of the $k$ like cells contains $\beta$ labelled compartments (these cells will be referred to as \textit{ordinary cells}),
		%	\item  the one unlike cell contains $\gamma$ labelled compartments (the cell will be referred to as the \textit{special cell}),
		\item  in each of the $k_i$ blocks, compartments have cyclic ordered numbering,
		\item in each of the $k_i$ blocks each compartment may be occupied by at most one element,
		\item on each of the $k_i$ blocks in each consecutive available $\alpha$ compartments only the first compartment may be occupied by an element,
		%	\item the first $r$ elements of $[n+r]$ are in different blocks,
		\item in each of the sections the $k_i+r$ blocks are linearly ordered such that the $r$ singletons are in different cycles, and none of the blocks are fixed where the standard form is used as point of reference, 
		\item where each $k_i$ running from 0 to $\lambda_i$,
		%\item $\lambda-1$ bars are inserted in-between the linearly ordered first $k+r$ blocks to form a barred preferential arrangement.
	\end{enumerate}
\item	the $(\lambda+1)th$ section has $\gamma$ labelled compartments, where the compartments have cyclic ordered numbering, and on each consecutive available $\alpha$ compartments only the first compartment may be occupied by an element.
	\end{itemize}
\end{definition}

We will now refer to the partitions in Definition~\ref{definition:270} as $(x,\alpha,\beta,\lambda,r)$-partitions.

\begin{theorem}For $\gamma,\beta,r,\lambda\in\mathbb{N}_0$, and $\alpha\in \mathbb{N}$, such that ${\alpha}|\beta$ and ${\alpha}|\gamma$ we have\label{theorem:33}\fontsize{10}{1}
	\begin{equation}
		\mathcal{B}^{r,x}_{n,\lambda}(\alpha,\beta,\gamma)=\sum\limits_{\lambda_1+\lambda_2+\cdots+\lambda_{\lambda+1}=n}\binom{n}{\lambda_1,\lambda_2,\ldots,\lambda_{\lambda+1}}(\gamma|\alpha)_{\lambda_{\lambda+1}}\prod_{i=1}^{\lambda}\sum_{k_i=0}^{\lambda_i}d_{k_i,r}x^{k_i}\beta^{k_i} S(\lambda_i,k_i,\alpha,\beta,\gamma).
	\end{equation}
\end{theorem}
\begin{proof} The $(x,\alpha,\beta,\lambda,r)$-partitions of $[n+r\lambda]$ can be formed in the following way. The number of elements that are in each of the $\lambda+1$ sections can be chosen in $\binom{n}{\lambda_1,\lambda_2,\ldots,\lambda_{\lambda+1}}$ ways. Given $\lambda_i+r$ elements in the $ith$ section. The elements can be arranged into $k_i+r$ subsets where $r$ fixed elements are singletons, and each of the $k_i$ subsets has $\beta$ compartments in $x^{k_i}\beta^{k_i} S(\lambda_i,k_i,\alpha,\beta,\gamma)$ ways. The $k_i+r$ subsets can be linearly ordered in $d_{k_i,r}$ ways, such that the $r$ singletons are in different cycles. Summing over $k_i$ gives the total number of ways of arranging the elements within the section. Hence, for all the $\lambda$ sections, the total number of ways of arranging the elements in them is $\prod_{i=1}^{\lambda}\sum_{k_i=0}^{\lambda_i}d_{k_i,r}x^{k_i}\beta^{k_i} S(\lambda_i,k_i,\alpha,\beta,\gamma)$. On the $(\lambda+1)th$ section, $\lambda_{\lambda+1}$ elements can be arranged among the $\gamma$ compartments in $(\gamma|\alpha)_{\lambda_{\lambda+1}}$ ways. Considering all possible solutions of the equation $\lambda_1+\lambda_2+\cdots+\lambda_{\lambda+1}=n$ in non-negative integers completes the proof.
\end{proof}

Theorem~\ref{theorem:33} can be thought as a deranged generalization of  Theorem~3.8 of \cite{nkonkobe2020combinatorial}.
%\begin{theorem}For %$\gamma,\beta,r,\lambda\in\mathbb{N}_0$, and $\alpha\in \mathbb{N}$, where ${\alpha}|\beta$ and ${\alpha}|\gamma$ we have
%	\begin{equation}
%		\mathcal{B}^{r,x}_{n,\lambda}(\alpha,\beta,\gamma)=\sum_{k=0}^{n}\begin{bmatrix}
%			\binom{k+\lambda-2}{k}+\binom{k+\lambda-2}{k-1}
%		\end{bmatrix}d_{k,r}\beta^k S(n,k,\alpha,\beta,\gamma).
%	\end{equation}
%\end{theorem}
%\begin{proof}
%	Combinatorial proof to be inserted.
%\end{proof}

\begin{theorem}\label{theorem:3}For $\gamma,\beta,r,\lambda\in\mathbb{N}_0$, and $\alpha\in \mathbb{N}$, where ${\alpha}|\beta$ and ${\alpha}|\gamma$ we have
	\begin{equation}\label{equation:2000}
		\mathcal{B}^{r,x}_{n,\lambda}(\alpha,\beta,\gamma)=\sum_{\lambda_1+\lambda_2+\cdots+\lambda_{\lambda+1}=n}\binom{n}{\lambda_1,\lambda_2,\ldots,\lambda_{\lambda+1}}(\gamma|\alpha)_{\lambda_{\lambda+1}}\prod_{j=1}^{\lambda}	\mathcal{B}^{r,x}_{\lambda_{j},1}(\alpha,\beta,0).
	\end{equation}
\end{theorem}
\begin{proof} In this theorem we view barred preferential arrangements as formed by first placing bars, and then forming blocks inbetween the bars (also noted in \cite{pippenger2010hypercube}). Given $\lambda$ bars they induce $\lambda+1$ sections. In a barred preferential arrangement we denote the number of element in the $jth$ section by $\lambda_j$. On each of the first $\lambda$ sections the number of ways of arranging the $\lambda_j+r$ elements within the section is $\mathcal{B}^{r,x}_{\lambda_{j},1}(\alpha,\beta,0)$ (see Definition~\ref{definition:5}). The reason that $\gamma$ is zero on $\mathcal{B}^{r,x}_{\lambda_{j},1}(\alpha,\beta,0)$ is that the $\gamma$ elements are on their own section on each of the barred preferential arrangements.  The number of ways of choosing elements going into the $\lambda+1$ sections is given by the multinomial coefficient $\binom{n}{\lambda_1,\lambda_2,\ldots,\lambda_{\lambda+1}}$. The $\lambda_{\lambda+1}$ elements within the $(\lambda+1)th$ section can be arranged among the $\gamma$ compartments in $(\gamma|\alpha)_{\lambda_{\lambda+1}}$ ways.
\end{proof}

For the special case $(\alpha,\beta,x)=(0,1,1)$ a variation of Theorem~\ref{theorem:3} has been discussed in (6) of \cite{nkonkobe2017study} or Theorem~3.1 of \cite{nkonkobe2020combinatorial}. Clearly \eqref{equation:2000} also follows from \eqref{equation:4} when we view \eqref{equation:4} as a convolution of $\lambda+1$ generating functions in the following way

\begin{equation}\label{equation:40}
	\sum\limits_{n=0}^{\infty}\mathcal{B}^{r,x}_{n,\lambda}(\alpha,\beta,\gamma)\frac{t^n}{n!}=(1+\alpha t)^{\frac{\gamma}{\alpha}}\prod_{i=1}^{\lambda}\begin{bmatrix}\frac{[x((1+\alpha t)^\frac{\beta}{\alpha}-1)]^{r}exp(-[x((1+\alpha t)^\frac{\beta}{\alpha}-1)])}{[1-x((1+\alpha t)^\frac{\beta}{\alpha}-1)]^{(r+1)}}\end{bmatrix}.
\end{equation}
%\begin{theorem}For %$\gamma,\beta,r,\lambda\in\mathbb{N}_0$, and $\alpha\in \mathbb{N}$, where ${\alpha}|\beta$ and ${\alpha}|\gamma$ we have
%	\begin{equation}
%		\omega _n^{(\lambda)}(x;{\alpha}, \beta, \gamma)=\sum\limits_{i=0}^{n}\binom{n}{i}\mathcal{B}^{0,x}_{i,\lambda}(\alpha,\beta,\gamma)\sum\limits_{l=0}^{n-i}\beta^lS(n-i,l,\alpha,\beta,0)x^l\lambda^{l}.
%	\end{equation}
%\end{theorem}
%\begin{proof}
	
	%We have
	
%	 \begin{align*}\sum\limits_{n=0}^{\infty}\begin{bmatrix}
	%	\sum\limits_{l=0}^{n}\beta^lS(n,l,\a%lpha,\beta,0)x^l\lambda^{n}
%	\end{bmatrix}\frac{t^n}{n!}&=\sum\limits_{l=0}^{\infty}\sum\limits_{n=0}^{\infty}
%	\beta^lS(n,l,\alpha,\beta,0)x^l\lambda^{l}
%\frac{t^n}{n!}\\&=exp(\lambda (x(1+\alpha t)^\frac{\beta}{\alpha}-1))\end{align*}

%This and \eqref{equation:4} complete the proof.
%\end{proof}
\begin{theorem}For $\gamma,\beta,\lambda\in\mathbb{N}_0$, and $\alpha\in \mathbb{N}$, where ${\alpha}|\beta$ and ${\alpha}|\gamma$ we have
	\begin{equation}
		\omega _{n}(x;{\alpha}, \beta, \gamma,\lambda)=\sum\limits_{i=0}^{n}\binom{n}{i}\mathcal{B}^{0,x}_{i,\lambda}(\alpha,\beta,\gamma)\sum\limits_{l=0}^{n-i}\beta^lS(n-i,l,\alpha,\beta,0)x^l\lambda^{l}.
	\end{equation}
\end{theorem}
\begin{proof}In view of Remark~\ref{remark:1} above, when one compares barred preferential arrangements having same blocks certain blocks can be interpreted as being fixed, and some being permuted. From $[n]$ we can choose $n-i$ elements in $\binom{n}{n-i}$ ways. These $(n-i)$ elements are to form the fixed blocks. The number of partitions of $[n-i]$ into $l$ blocks each having $\beta$ compartments satisfying the conditions of Definition~\ref{definition:10} is $\beta^lS(n-i,l,\alpha,\beta,0)$. The $l$ blocks will constitute the fixed blocks in the sections. The $l$ blocks can be colored in $x^l$ ways. The $l$ blocks can then be placed into $\lambda$ sections in $\lambda^l$ ways, where $l$ ranges from 0 to $n-i$. Once these blocks are placed on the respective sections they are fixed. With the remaining $i$ elements, arrangements can be made in which none of the blocks placed inbetween the bars are fixed and the block having the $\gamma$ compartments maybe none empty this can be done in  $\mathcal{B}^{0,x}_{i,\lambda}(\alpha,\beta,\gamma)$ ways.
	
\end{proof}

\iffalse
\begin{theorem}For $\gamma,\beta,r,\lambda\in\mathbb{N}_0$, and $\alpha\in \mathbb{N}$, where ${\alpha}|\beta$ and ${\alpha}|\gamma$ we have
	\begin{equation}
		\omega _n^{(\lambda)}(x;{\alpha}, \beta, \gamma)=\sum\limits_{i=0}^{n}\binom{n}{i}\mathcal{B}^{0,x}_{i,\lambda}(\alpha,\beta,\gamma)\sum\limits_{l=0}^{n-i}\beta^lS(n-i,l,\alpha,\beta,0)x^l\lambda^{l}.
	\end{equation}
\end{theorem}
\begin{proof}In each barred preferential arrangement certain blocks are fixed and some are permuted. From $[n]$ we can choose $n-i$ elements in $\binom{n}{n-i}$ ways. These elements will be used to form the fixed blocks. The number of partitions of $[n-i]$ into $l$ blocks each having $\beta$ compartments satisfying the conditions of Definition~\ref{definition:11} is $\beta^lS(n-i,l,\alpha,\beta,0)$, where the block having $\gamma$ compartments is empty. The $l$ cells can be colored in $x^l$ ways. The $l$ blocks can then be placed into $\lambda$ sections in $\lambda^l$ ways, where $l$ ranges from 0 to $n-i$. Once these blocks are placed on the respective sections they are fixed. With the remaining $i$ elements, arrangements can be made in which none of the blocks placed inbetween the bars are fixed and the block having the $\gamma$ compartments maybe none empty this can be done in  $\mathcal{B}^{0,x}_{i,\lambda}(\alpha,\beta,\gamma)$ ways.
	
\end{proof}
\fi

\section{Asymptotics\label{section:4}}
In this section, we propose asymptotic results for the numbers $\mathcal{B}^{r,x}_{n,\lambda}(\alpha,\beta,\gamma)$ based on a method developed in \cite{hsu1990power,hsu1991kind}. The method has also been used in \cite{JoseNkonkobeunified,corcino2020second,hsu1998unified,nkonkobe2020combinatorial}.  Let  $\overline{k}$ denote the vector $(k_1,k_2,\ldots,k_n)$  representing the partition $1^{k_1}2^{k_2}\cdots n^{k_n}$, such that $k_1+2k_2+\cdots+nk_n=n$. Where $k=k_1+k_2+\cdots+k_n$ represents the number of parts of the partition $1^{k_1}2^{k_2}\cdots n^{k_n}$. Let $c\in\mathbb{C}$ and $n\in\mathbb{N}$. Let $(c)_n$ denote the product $c(c-1)(c-2)\cdots(c-n+1)$, with $(c)_0=1$. Let  $\sigma(n)$ denote the set of partitions of the integer $n\in \mathbb{N}$, and  $\sigma(n,k)$ denote the set of partitions of the integer $n\in \mathbb{N}$ having $k$ parts. Let  $\Omega(t)$ denote the formal power series $\sum\limits_{n=0}^\infty b_nt^n$ over $\mathbb{C}$ with $\Omega(0)=b_0=1$. We suppose that, for any $\alpha\in \mathbb{C}$ with $\alpha\not= 0$, we have a formal power series

\begin{equation}
	\Gamma(t)=(\Omega(t))^\alpha=\sum_{n=0}^\infty a(\alpha,n)t^n,\;
\end{equation} where $a(\alpha,n)=[t^n]\Gamma(t)$ and ${\alpha\brace 0}=1$. Then for $\delta\in \mathbb{C}$ with $\delta\not=0$ we have (see \cite{hsu1990power,hsu1991kind})

\begin{equation}\label{equation:77}\frac{1}{(\delta)_n}\left[t^n\right]\left(\Gamma(t)\right)^{\delta}=\frac{1}{(\delta)_n}a(\alpha\delta,n)=\sum_{f=0}^m\frac{W(n,f)}{(\delta-n+f)_f}\;\;+o\left(\frac{W(n,m)}{(\delta-n+m)_m}\right),\end{equation}
where $W(n,f)=\sum\limits_{\sigma(n,n-f)}\frac{b_1^{k_1}b_2^{k_2}\cdots b_n^{k_n}}{k_1!k_2!\cdots k_n!}$, $0\leq i<n$,  and $n=o(\sqrt{|\delta|})$ as  $|\delta|\rightarrow \infty$. By \eqref{equation:77} we have the following result.

\begin{theorem}\label{theorem:30}\cite{hsu1990power,hsu1991kind} For all integers $n\geq0$, we have

	\begin{equation}
		\frac{a(\alpha\delta,n)}{(\delta)_n}=\sum\limits_{f=0}^m\frac{W(n,f)}{(\delta-n+f)_f}+o\begin{pmatrix}
			\frac{W(n,m)}{(\delta-n+m)_m}
		\end{pmatrix}.
	\end{equation}
	\end{theorem}

Applying Theorem~\ref{theorem:30} to \eqref{equation:4} we have the following result. 
\begin{corollary}\label{theorem:200} For all integers $n\geq0$, we have

	\begin{equation}
		\frac{\mathcal{B}^{r,x}_{n,\lambda}(\alpha,\beta,\gamma\lambda)}{(\delta)_nn!}=\sum\limits_{f=0}^m\frac{W(n,f)}{(\delta-n+f)_f}+o\begin{pmatrix}
			\frac{W(n,m)}{(\delta-n+m)_m}
		\end{pmatrix},
	\end{equation}
	
	\noindent	where \begin{equation}\label{equation:200}W(n,f)=\sum\limits_{\sigma(n,n-f)}\prod\limits_{i=1}^n\frac{1}{k_i!}\begin{bmatrix}
			\frac{\mathcal{B}^{r,x}_{i,1}(\alpha,\beta,\gamma)}{i!}
		\end{bmatrix}^{k_i},\end{equation} and $n=o(\sqrt{|\delta|})$ as  $|\delta|\rightarrow \infty$.
	
\end{corollary}

We now compute few values of $W(n,f)$.
\begin{align*}W(n,0)&=\frac{1}{n!}\begin{Bmatrix}\frac{\mathcal{B}^{r,x}_{1,1}(\alpha,\beta,\gamma)}{1!}\end{Bmatrix}^{n},\\
W(n,1)&=\frac{1}{(n-2)!}\begin{Bmatrix}\frac{\mathcal{B}^{r,x}_{1,1}(\alpha,\beta,\gamma)}{1!}\end{Bmatrix}^{n-2}\begin{Bmatrix}\frac{\mathcal{B}^{r,x}_{2,1}(\alpha,\beta,\gamma)}{2!}\end{Bmatrix},\end{align*}
\begin{align*}
W(n,2)=&\frac{1}{(n-3)!}\begin{Bmatrix}\frac{\mathcal{B}^{r,x}_{1,1}(\alpha,\beta,\gamma)}{1!}\end{Bmatrix}^{n-3}\begin{Bmatrix}\frac{\mathcal{B}^{r,x}_{3,1}(\alpha,\beta,\gamma)}{3!}\end{Bmatrix}\\&+\frac{1}{2!(n-4)!}\begin{Bmatrix}\frac{\mathcal{B}^{r,x}_{1,1}(\alpha,\beta,\gamma)}{1!}\end{Bmatrix}^{n-4}\begin{Bmatrix}\frac{\mathcal{B}^{r,x}_{2,1}(\alpha,\beta,\gamma)}{2!}\end{Bmatrix}^2,\\
	W(n,3)=&\frac{1}{(n-4)!}\begin{Bmatrix}
		\frac{\mathcal{B}^{r,x}_{1,1}(\alpha,\beta,\gamma)}{1!}
	\end{Bmatrix}^{n-4}\begin{Bmatrix}
		\frac{\mathcal{B}^{r,x}_{4,1}(\alpha,\beta,\gamma)}{4!}
	\end{Bmatrix}\\&+\frac{1}{(n-5)!}\begin{Bmatrix}
		\frac{\mathcal{B}^{r,x}_{1,1}(\alpha,\beta,\gamma)}{1!}
	\end{Bmatrix}^{n-5}\begin{Bmatrix}
		\frac{\mathcal{B}^{r,x}_{2,1}(\alpha,\beta,\gamma)}{2!}
	\end{Bmatrix}\begin{Bmatrix}
		\frac{\mathcal{B}^{r,x}_{3,1}(\alpha,\beta,\gamma)}{3!}
	\end{Bmatrix}\\&+\frac{1}{3!(n-6)!}\begin{Bmatrix}
		\frac{\mathcal{B}^{r,x}_{1,1}(\alpha,\beta,\gamma)}{1!}
	\end{Bmatrix}^{n-6}\begin{Bmatrix}
		\frac{\mathcal{B}^{r,x}_{2,1}(\alpha,\beta,\gamma)}{2!}
	\end{Bmatrix}^3,
\end{align*}

\begin{align*}W(n,4)=&\frac{1}{(n-5)!}\begin{Bmatrix}
		\frac{\mathcal{B}^{r,x}_{1,1}(\alpha,\beta,\gamma)}{1!}
	\end{Bmatrix}^{n-5}\begin{Bmatrix}
		\frac{\mathcal{B}^{r,x}_{5,1}(\alpha,\beta,\gamma)}{5!}
	\end{Bmatrix}\\&+\frac{1}{2!(n-6)!}\begin{Bmatrix}
		\frac{\mathcal{B}^{r,x}_{1,1}(\alpha,\beta,\gamma)}{1!}
	\end{Bmatrix}^{n-6}\begin{Bmatrix}
		\frac{\mathcal{B}^{r,x}_{3,1}(\alpha,\beta,\gamma)}{3!}
	\end{Bmatrix}^2\\&+\frac{1}{2!(n-7)!}\begin{Bmatrix}
		\frac{\mathcal{B}^{r,x}_{1,1}(\alpha,\beta,\gamma)}{1!}
	\end{Bmatrix}^{n-7}\begin{Bmatrix}
		\frac{\mathcal{B}^{r,x}_{2,1}(\alpha,\beta,\gamma)}{2!}
	\end{Bmatrix}^2\begin{Bmatrix}
		\frac{\mathcal{B}^{r,x}_{1,1}(\alpha,\beta,\gamma)}{3!}
	\end{Bmatrix}\\&+\frac{1}{4!(n-8)!}\begin{Bmatrix}
		\frac{\mathcal{B}^{r,x}_{1,1}(\alpha,\beta,\gamma)}{1!}
	\end{Bmatrix}^{n-8}\begin{Bmatrix}
		\frac{\mathcal{B}^{r,x}_{2,1}(\alpha,\beta,\gamma)}{2!}
	\end{Bmatrix}^4\\&+
	\frac{1}{2!(n-6)!}\begin{Bmatrix}
		\frac{\mathcal{B}^{r,x}_{1,1}(\alpha,\beta,\gamma)}{1!}
	\end{Bmatrix}^{n-6}\begin{Bmatrix}
		\frac{\mathcal{B}^{r,x}_{2,1}(\alpha,\beta,\gamma)}{2!}
	\end{Bmatrix}\begin{Bmatrix}
		\frac{\mathcal{B}^{r,x}_{4,1}(\alpha,\beta,\gamma)}{4!}
	\end{Bmatrix},
\end{align*}
\begin{align*}
	W(n,5)=&\frac{1}{(n-6)!}\begin{Bmatrix}
		\frac{\mathcal{B}^{r,x}_{1,1}(\alpha,\beta,\gamma)}{1!}
	\end{Bmatrix}^{n-6}\begin{Bmatrix}
		\frac{\mathcal{B}^{r,x}_{6,1}(\alpha,\beta,\gamma)}{6!}
	\end{Bmatrix}\\&+\frac{1}{(n-7)!}\begin{Bmatrix}
		\frac{\mathcal{B}^{r,x}_{1,1}(\alpha,\beta,\gamma)}{1!}
	\end{Bmatrix}^{n-7}\begin{Bmatrix}
		\frac{\mathcal{B}^{r,x}_{2,1}(\alpha,\beta,\gamma)}{2!}
	\end{Bmatrix}\begin{Bmatrix}
		\frac{\mathcal{B}^{r,x}_{5,1}(\alpha,\beta,\gamma)}{5!}
	\end{Bmatrix}\\&+\frac{1}{(n-7)!}\begin{Bmatrix}
		\frac{\mathcal{B}^{r,x}_{1,1}(\alpha,\beta,\gamma)}{1!}
	\end{Bmatrix}^{n-7}\begin{Bmatrix}
		\frac{\mathcal{B}^{r,x}_{4,1}(\alpha,\beta,\gamma)}{4!}
	\end{Bmatrix}\begin{Bmatrix}
		\frac{\mathcal{B}^{r,x}_{3,1}(\alpha,\beta,\gamma)}{3!}
	\end{Bmatrix}\\&+\frac{1}{2!(n-8)!}\begin{Bmatrix}
		\frac{\mathcal{B}^{r,x}_{1,1}(\alpha,\beta,\gamma)}{1!}
	\end{Bmatrix}^{n-8}\begin{Bmatrix}
		\frac{\mathcal{B}^{r,x}_{2,1}(\alpha,\beta,\gamma)}{2!}
	\end{Bmatrix}^2\\&+
	\frac{1}{2!(n-8)!}\begin{Bmatrix}
		\frac{\mathcal{B}^{r,x}_{1,1}(\alpha,\beta,\gamma)}{1!}
	\end{Bmatrix}^{n-8}\begin{Bmatrix}
		\frac{\mathcal{B}^{r,x}_{2,1}(\alpha,\beta,\gamma)}{2!}
	\end{Bmatrix}\begin{Bmatrix}
		\frac{\mathcal{B}^{r,x}_{3,1}(\alpha,\beta,\gamma)}{3!}
	\end{Bmatrix}^2\\&+
	\frac{1}{3!(n-9)!}\begin{Bmatrix}
		\frac{\mathcal{B}^{r,x}_{1,1}(\alpha,\beta,\gamma)}{1!}
	\end{Bmatrix}^{n-9}\begin{Bmatrix}
		\frac{\mathcal{B}^{r,x}_{2,1}(\alpha,\beta,\gamma)}{2!}
	\end{Bmatrix}^3\begin{Bmatrix}
		\frac{\mathcal{B}^{r,x}_{3,1}(\alpha,\beta,\gamma)}{3!}
	\end{Bmatrix}\\&+
	\frac{1}{5!(n-10)!}\begin{Bmatrix}
		\frac{\mathcal{B}^{r,x}_{1,1}(\alpha,\beta,\gamma)}{1!}
	\end{Bmatrix}^{n-10}\begin{Bmatrix}
		\frac{\mathcal{B}^{r,x}_{2,1}(\alpha,\beta,\gamma)}{2!}
	\end{Bmatrix}^5.
\end{align*} 

Thus, 
\begin{align*}\frac{\mathcal{B}^{r,x}_{n,\lambda}(\alpha,\beta,\gamma)}{n!}\sim\: & (\delta)_nW(n,0)+(\delta)_{n-1}W(n,1)+(\delta)_{n-2}W(n,2)\\&+(\delta)_{n-3}W(n,3)+(\delta)_{n-4}W(n,4)+(\delta)_{n-5}W(n,5).\end{align*}

\section{Conflict of Interest} The author indicated no conflict of interest.
\section{Conclusion and Further works}
In this study we proposed a combinatorial interpretation of higher order $r$-deranged Bell numbers, including interpretations of some of their combinatorial identities. We also provided some asymptotic results. The work done in this study may be adapted to provide combinatorial interpretation of other kinds of Bell like polynomials. This can be achieve by adapting the combinatorial interpretation we have given in this study of the numbers $\mathcal{B}^{r,x}_{n,\lambda}(\alpha,\beta,\gamma)$, as the underlying stirling numbers using in $\mathcal{B}^{r,x}_{n,\lambda}(\alpha,\beta,\gamma)$ generalize several other combinatorial numbers. For future research one may also study the properties of the roots of the numbers $\mathcal{B}^{r,x}_{n,\lambda}(\alpha,\beta,\gamma)$ we have defined in this study. Recently in \cite{djemmada2025partial} the authors discussed non-degenerate partial deranged bell numbers, which represent the number of partitions of a set $[n]$ with exactly $r$ fixed blocks, incoporating our results in this study one may extend this to degenerate polynomials as well.  %Higher order $r$ Deranged polynomials based on these mentioned generalisations of stirling numbers may be studied. Roots of such polynomials may be studied. 
% The Higher order $r$-deranged Bell polynomials may also

\section{Acknowledgements} The author would like to thank the anonymous referees for their valuable comments. The author would like to also acknowledge financial support from university of the Witwatersrand through the FR\&IC Start-up Research Grant.

\bibliography{document.bib}{}
\bibliographystyle{plain}

\iffalse

 \fi

\end{document}